\documentclass{article}

\usepackage{amsmath,amssymb,amsthm,enumitem,mathrsfs}
\usepackage[dvipdfmx]{graphicx}

\usepackage{mathtools}
  \mathtoolsset{showonlyrefs}
  \numberwithin{equation}{section}
  \numberwithin{figure}{section}
  \numberwithin{table}{section}

\newtheorem{theorem}{Theorem}[section]
\newtheorem{lemma}[theorem]{Lemma}
\newtheorem{proposition}[theorem]{Proposition}

\newtheorem*{assumption}{Assumption}
\newtheorem*{notation}{Notation}
\newtheorem*{remark}{Remark}

\def\Xint#1{\mathchoice
    {\XXint\displaystyle\textstyle{#1}}%
    {\XXint\textstyle\scriptstyle{#1}}%
    {\XXint\scriptstyle\scriptscriptstyle{#1}}%
    {\XXint\scriptscriptstyle\scriptscriptstyle{#1}}%
      \!\int}
\def\XXint#1#2#3{{\setbox0=\hbox{$#1{#2#3}{\int}$}
    \vcenter{\hbox{$#2#3$}}\kern-.5\wd0}}
\def\dashint{\Xint-}
\DeclareMathOperator*{\argmin}{argmin}
\DeclareMathOperator{\tr}{tr}
\DeclareMathOperator{\Sym}{Sym}

\begin{document}
  \thispagestyle{plain}
  \title{Least squares estimators based on the Adams method for stochastic differential equations with small L\'evy noise}
  \author{Mitsuki Kobayashi \\
          Department of Pure and Applied Mathematics, Waseda University,\\
          3-4-1 Ohkubo, Shinjuku-ku, Tokyo 169-8555, JAPAN \\
          (E-mail: mitsuki@fuji.waseda.jp) \\
          and \\
          Yasutaka Shimizu \\
          Department of Applied Mathematics, Waseda University, \\
          3-4-1 Ohkubo, Shinjuku-ku, Tokyo 169-8555, JAPAN \\
          (E-mail: shimizu@waseda.jp)}
  \date{}
  \maketitle

\begin{abstract}
  We consider stochastic differential equations (SDEs) driven by small L\'evy noise with some unknown parameters, and propose a new type of least squares estimators based on discrete samples from the SDEs. To approximate the increments of a process from the SDEs, we shall use not the usual Euler method, but the Adams method, that is, a well-known numerical approximation of the solution to the ordinary differential equation appearing in the limit of the SDE. We show the consistency of the proposed estimators as well as the asymptotic distribution in a suitable observation scheme.
  We also show that our estimators can be better than the usual LSE based on the Euler method in the finite sample performance.\\


\end{abstract}

\section{Introduction}

  This paper is concerned with the following $\mathbb{R}^d$-valued stochastic differential equation
  \begin{equation}\label{eq: main SDE}
    dX_t^{\varepsilon,\theta} = b(X_t^{\varepsilon,\theta},\theta) \, dt + \varepsilon \, dL_t,
    \quad X_0^{\varepsilon,\theta} = x_0 \in \mathbb{R}^d,
  \end{equation}
  where $\Theta_0$ is a smooth bounded open convex set in $\mathbb{R}^p$ with $p\in\mathbb{N}$,
  $\Theta$ denotes the closure of $\Theta_0$,
  $\theta\in\Theta$,
  $\varepsilon>0$,
  $b$ is a function from $\mathbb{R}^d\times\Theta$ to $\mathbb{R}^d$,
  and $L=(L_t)_{t\leq0}$ is a $d$-dimensional L\'{e}vy process given by
  \begin{equation*}
    L_t = a t + \sigma B_t + \int_0^t \int_{|z|\leq1} z \, \tilde{N}(ds,dz) + \int_0^t \int_{|z|>1} z \, N(ds,dz)
  \end{equation*}
  with $a\in\mathbb{R}^d$, a $d\times r$ real-valued matrix $\sigma$, an $r$-dimensional standard Brownian motion $B_t$,
  an independent Poisson random measure $N(ds,dz)$ with characteristic measure $dt\,\nu(dz)$,
  and a martingale measure $\tilde{N}(ds,dz)=N(ds,dz)-\nu(dz)ds$.
  Here, we assume that $\nu(dz)$ is a L\'evy measure on $\mathbb{R}^d\setminus\{0\}$ and $\int_{|z|>0} |z| \,\nu(dz)<\infty$.
  Suppose that we have discrete data $X_{t_0}^{\varepsilon}, \dots, X_{t_n}^{\varepsilon}$ from \eqref{eq: main SDE} under $\theta = \theta_0 \in \Theta_0$ with $X_t^\varepsilon := X_t^{\varepsilon,\theta_0}$,
  and that $0=t_0<\dots<t_n=1$ and $t_i - t_{i-1} = 1/n$.
  We consider the problem of estimating the true $\theta_0\in\Theta_0$ under $n\to\infty$ and $\varepsilon\to0$ at the same time.
  We also define $x_t$ as the solution of the corresponding deterministic differential equation
  \begin{equation}\label{eq: main ODE}
    \frac{dx_t}{dt} = b(x_t, \theta_0)
  \end{equation}
  with the initial condition $x_0$.

  Problems of parametric estimation for discretely observed stochastic processes with small diffusion have been studied by various authors (\textit{e.g.}, Genon-Catalot \cite{genon-catalot1990maximum}, Laredo \cite{laredo1990sufficient}, S{\o}rensen and Uchida \cite{sorensen2003small} and so on) and problems of ones with small L\'evy noise have been studied by Long \textit{et al.} \cite{long2013least},
  Long \textit{et al.} \cite{long2017least} and references therein, while the performance of such estimators become better when ‘large shocks’ due to noise are truncated (see Shimizu \cite{shimizu2017threshold}).

  Before constructing our LSEs, let us introduce the well-known Adams method in numerical analysis for ODEs (see, \textit{e.g.}, Butcher \cite{butcher2016numerical}, Hairer \textit{et al.} \cite{hairer1993solving}, Hairer and Wanner \cite{hairer2010solving}
  and Iserles \cite{iserles2008}),
  which is the combinations of two methods as preditor-corrector pair, says, the Adams-Bashforth and the Adams-Moulton formulae. For instance, to compute an approximate value $\hat{x}_{t_k}$ of the solution of \eqref{eq: main ODE} at $t=t_k$, we firstly prepare a predictor $x^*_{t_k}$ given by Adams-Bashforth method with $\ell=1,2,\dots$ as
  \begin{equation}
    x_{t_k}^* = \hat{x}_{t_{k-1}} + \frac{1}{n} \sum_{\nu=1}^{\ell} \gamma_{\ell\nu} b(\hat{x}_{t_{k-\nu}},\theta_0), \quad
    \gamma_{\ell\nu} := \frac{(-1)^{\nu-1}}{(\nu-1)!(\ell-\nu)!} \int_0^1 \prod_{\substack{j=1 \\  j\neq\nu}}^\ell (u+j-1) \, du, \label{eq: def of adams-bashforth}
  \end{equation}
  by using the past approximate values $\hat{x}_{t_{k-1}},\dots,\hat{x}_{t_{k-\ell}}$ with $\hat{x}_{t_0}=x_{0}$,
  and we secondly modify the value $x^*_{t_k}$ to a corrector $\hat{x}_{t_k}$ given by Adams-Moulton method as
  \begin{equation}
    \hat{x}_{t_k} = \hat{x}_{t_{k-1}} + \frac{1}{n}\beta_{\ell0} b(x_{t_k}^*,\theta_0) + \frac{1}{n} \sum_{\nu=1}^{\ell} \beta_{\ell\nu} b(\hat{x}_{t_\nu},\theta_0), \quad
    \beta_{\ell\nu} := \frac{(-1)^{\nu}}{\nu!(\ell-\nu)!} \int_0^1 \prod_{\substack{j=0 \\  j\neq\nu}}^\ell (u+j-1) \, du. \label{eq: def of adams-moulton}
  \end{equation}
  Both formulae follows by the same argument as in Section 2.1 in Iserles \cite{iserles2008},
  and the predictor-corrector scheme is written in Hairer and Wanner \cite{hairer2010solving}.
  Some of the values of the coefficients $\gamma_{\ell\nu}$, $\beta_{\ell\nu}$ can be seen in Table 244 in Butcher \cite{butcher2016numerical}.
  Here, we remark that for any $g:\mathbb{R}\to\mathbb{R}$, the coefficients $\gamma_{\ell\nu}$ and $\beta_{\ell\nu}$ satisfy
  \begin{align}
    &\int_{t_{k-1}}^{t_k} P(s;g,t_{k-1},\dots,t_{k-\ell}) \, ds = \frac{1}{n} \sum_{\nu=1}^{\ell} \gamma_{\ell\nu} g(x_{t_{k-\nu}},\theta),\\
    &\int_{t_{k-1}}^{t_k} P(s;g,t_k,\dots,t_{k-\ell}) \, ds = \frac{1}{n} \sum_{\nu=0}^{\ell} \beta_{\ell\nu} g(x_{t_{k-\nu}},\theta), \label{eq: polynomial}
  \end{align}
  where $s\mapsto P(s;g,t_k,\dots,t_{k-\ell})$ is the Lagrange interpolating polynomial through the points $(s,g(s))$, $s=t_k,\dots,t_{k-\ell}$ (see, \textit{e.g.}, Section III.1 in Hairer \textit{et al.} \cite{hairer1993solving}).
  In particular, substituting $g\equiv 1$, we have
  \begin{equation}
    \sum_{\nu=1}^\ell \gamma_{\ell\nu} = \sum_{\nu=0}^\ell \beta_{\ell\nu} = 1.
  \end{equation}

  The Euler method sometimes fails to approximate the solution of ODEs
  (\textit{e.g.}, $b(x,\theta)=-\theta x$ for $x,\theta>0$ and $\theta/n\notin(0,2)$, in Section 4.2 in Iserles \cite{iserles2008}), and is less accurate than the Runge-Kutta method, the Adams method, etc.
  For linearity and simplicity, we employ the Adams method and define
  the \textit{Adams-Moulton type} contrast function $\Psi_{n,\varepsilon,\ell}(\theta)$ as
  \begin{equation}
    \Psi_{n,\varepsilon,\ell}(\theta) := \sum_{k=\ell\vee1}^{n} \frac{ \left | X_{t_k}^{\varepsilon} - X_{t_{k-1}}^{\varepsilon} - \frac{1}{n} A_\ell b(\pmb{X}_{t_k:t_{k-\ell}}^{\varepsilon}, \theta) \right |^2 }{\varepsilon^2/n},
  \end{equation}
  where $\pmb{X}_{t_k:t_{k-\ell}}^\varepsilon := (X_{t_k}^\varepsilon, \dots, X_{t_{k-\ell}}^\varepsilon)$
  and $A_\ell$ is the operator from $C^{\ell+1}(\mathbb{R}^d;\mathbb{R}^d)$ to $C^{\ell+1}(\mathbb{R}^{d\times\ell};\mathbb{R}^d)$ of the form
  \begin{equation}
    A_\ell f(\pmb{x}) := \sum_{\nu=0}^{\ell} {\beta_{\ell\nu}} f(x_\nu) \qquad \text{for}~\pmb{x}=(x_0,\dots,x_{\ell}) \in \mathbb{R}^{d\times\ell},~\theta\in\Theta,
  \end{equation}
  in particular,
  \begin{equation}
    A_\ell b(\pmb{X}_{t_k:t_{k-\ell}}^{\varepsilon}, \theta) = \sum_{\nu=0}^{\ell} {\beta_{\ell\nu}} b(X_{t_{k-\nu}},\theta). \label{eq: adams-moulton sum}
  \end{equation}
  For simplicity of discussion, it is useful to use the following form for the contrast function
  \begin{equation}
    \Phi_{n,\varepsilon,\ell}(\theta) := \varepsilon^2 (\Psi_{n,\varepsilon,\ell}(\theta) - \Psi_{n,\varepsilon,\ell}(\theta_0)).
  \end{equation}
  Then the LSE is given by
  \begin{equation}\label{eq: def of estimator theta_n}
    \hat{\theta}_{n,\varepsilon,\ell} := \argmin_{\theta\in\Theta} \Psi_{n,\varepsilon,\ell}(\theta)
    = \argmin_{\theta\in\Theta} \Phi_{n,\varepsilon,\ell}(\theta).
  \end{equation}
  Similarly, we denote by $\tilde{\Psi}_{n,\varepsilon,\ell}$ the \textit{Adams-Bashforth type} contrast function
  \begin{equation}
    \tilde{\Psi}_{n,\varepsilon,\ell}(\theta) := \sum_{k=\ell}^{n} \frac{ \left | X_{t_k}^{\varepsilon} - X_{t_{k-1}}^{\varepsilon} - \frac{1}{n} \tilde{A}_\ell b(\pmb{X}_{t_{k-1}:t_{k-\ell}}^{\varepsilon}, \theta) \right |^2 }{\varepsilon^2/n},
  \end{equation}
  where
  \begin{equation}
    \tilde{A}_\ell b(\pmb{X}_{t_{k-1}:t_{k-\ell}}^{\varepsilon}, \theta) = \sum_{\nu=1}^{\ell} {\gamma_{\ell\nu}} b(X_{t_{k-\nu}},\theta). \label{eq: adams-bashforth sum}
  \end{equation}
  Then the LSE $\tilde{\theta}_{n,\varepsilon,\ell}$ is given by
  \begin{equation}\label{eq: def of estimator theta_n via adams-bashforth}
    \tilde{\theta}_{n,\varepsilon,\ell} := \argmin_{\theta\in\Theta} \tilde{\Psi}_{n,\varepsilon,\ell}(\theta).
  \end{equation}
  We call $\hat{\theta}_{n,\varepsilon,\ell}$ and $\tilde{\theta}_{n,\varepsilon,\ell}$ the \textit{Adams-Moulton type} LSE and the \textit{Adams-Bashforth type} LSE, respectively.

  \begin{notation}
    The following notations will be needed throughout the paper:
    \begin{align}
      &\bullet~\mathbb{N}_0:=\mathbb{N}\cup\{0\},~B_M\subset\mathbb{R}^d~\text{is a closed ball centered at the origin with radius $M>0$}.\\
      &\bullet~C^{\infty,0}(\mathbb{R}^d\times\Theta;\mathbb{R}^d)
        := \left \{ f:\mathbb{R}^d\times\Theta\to\mathbb{R}^d \, \middle | \, \begin{aligned}
          &\text{$f$ is smooth with respet to $x\in\mathbb{R}^d$, and for}\\
          &\text{all $k\in\mathbb{N}$, the $k$-th derivatives of $f$ with}\\
          &\text{respect to $x\in\mathbb{R}^d$ are continuous on $\mathbb{R}^d\times\Theta$}
        \end{aligned} \right \}.\\
      &\bullet~\partial_{\theta_j}:={\textstyle \frac{\partial}{\partial \theta_j}}~\text{with $j=1,\dots,p,$}~
        D_x^\alpha:={\textstyle \frac{\partial^{|\alpha|}}{\partial x_1^{\alpha_1}\dots\partial x_d^{\alpha_d}}}~
        \text{with $\alpha\in\mathbb{N}_0^d$, $|\alpha|=\alpha_1+\cdots+\alpha_d$.}\\
      &\begin{aligned}
        \:\bullet~
          &\left \| f \right \|_{C^{\infty,0}(B_M\times\Theta)}
            := \sup_{\alpha\in\mathbb{N}_0^d} \left \| D_x^\alpha f \right \|_{C(B_M\times\Theta)}
            = \sup_{\alpha\in\mathbb{N}_0^d} \sup_{(x,\theta)\in B_M\times\Theta} \left | D_x^\alpha f (x,\theta) \right |, \\
          &\left \| f(\theta_0) \right \|_{C^\infty(B_M)}
            := \sup_{\alpha\in\mathbb{N}_0^d} \left \| D_x^{\alpha} f (\cdot,\theta_0) \right \|_{C(B_M)}
            = \sup_{\alpha\in\mathbb{N}_0^d} \sup_{x\in B_M} \left | D_x^{\alpha} f (x,\theta_0) \right |, \\
          &\text{where}~f\in C^{\infty,0}(\mathbb{R}^d\times\Theta;\mathbb{R}^d),~M>0.
      \end{aligned}\\
      &\textstyle \bullet~\dashint_{t_{k-1}}^{t_k} f(t) \, dt ~\text{denotes the average integral}~ \frac{1}{|t_{k}-t_{k-1}|} \int_{t_{k-1}}^{t_k} f(t) \, dt.\\
      &\bullet~Y^{n,\varepsilon}_t:=X^{\varepsilon}_{\lceil nt\rceil/n}~\text{for $t\in(-1/n,1]$, where $\lceil\cdot\rceil$ is the ceiling function}.\\
      &\bullet~\|\sigma\|_{F}^2:=\tr(\sigma^T \sigma)=\sum_{ij} \sigma_{ij}^2,~\text{where $\sigma=(\sigma_{ij})$ is a $d\times r$ matrix}.
    \end{align}
  \end{notation}

  \begin{assumption}
    We will make the following assumptions:
    \begin{enumerate}[label={\rm (A\arabic*)}]
      \item \label{asmp: 2} The family $\{b(\cdot,\theta)\}_{\theta\in\Theta}$ is \textit{equi-Lipschitz continuous}, \textit{i.e.}, there is a postive constant $C$ called a \textit{common Lipschitz constant} such that
      \begin{equation}
        \left | b(x,\theta) - b(y,\theta) \right | \leq C \left | x-y \right | \quad (x,y\in\mathbb{R}^d,~\theta\in\Theta).
      \end{equation}
      \item \label{asmp: 4} The function $b$ belongs to $C^{\infty,0}(\mathbb{R}^d\times\Theta;\mathbb{R}^d)$,
      and $\left \| b \right \|_{C^{\infty,0}(B_M\times\Theta)}<\infty$ for all $M>0$.
      \item \label{asmp: 3} The function $b$ is differentiable with respect to $\theta\in\Theta_0$, and the families $\left \{ \partial_{\theta_j} b (\cdot,\theta) \right \}_{\theta\in\Theta_0}$ $(j=1,\dots,p)$ are equi-Lipschitz continuous.
      \item \label{asmp: 5} If $\theta\neq\theta_0$, then $b(x_t,\theta)\neq b(x_t,\theta_0)$ for some $t\in[0,1]$.
    \end{enumerate}
  \end{assumption}

\section{Convergence}
  \begin{proposition}\label{prop: Yt to Xt}
    Suppose the assumption \ref{asmp: 2}.
    \begin{enumerate}[label={\rm (\roman*)}]
      \item \label{3.1.i} It holds that
            \begin{equation}
              \sup_{\substack{\nu=0,\dots,\ell \\ t\in(t_{(\ell-1)\vee0},1]}} \left | Y_{t-t_\nu}^{n,\varepsilon} - x_t \right |
              \leq C \left ( \varepsilon \sup_{s\in[0,1]} \left | L_s \right | + \frac{\ell+1}{n} \right ),
            \end{equation}
            where $C$ is a positive constant, and $Y^{n,\varepsilon}_t:=X^{\varepsilon}_{\lceil nt\rceil/n}$ with the ceiling function $\lceil\cdot\rceil$.
      \item \label{3.1.ii} Let $\ell=\ell_n$ depend on $n$. If $\ell/n=O(1)$ as $n,\ell\to\infty$, then
            \begin{equation}
              \sup_{\substack{0<\varepsilon<1 \\ n\in\mathbb{N}}} \sup_{\substack{\nu=0,\dots,\ell \\ t\in(t_{(\ell-1)\vee0},1]}} \left | Y_{t-t_\nu}^{n,\varepsilon} \right | < \infty \quad \text{a.s.},
            \end{equation}
            and
            \begin{equation}
              \tau_m^{n,\varepsilon,\ell} := \inf \left \{ t>0 \, \Big | \, \left | x_t \right | \geq m,
              \min_{\nu=0,\dots,\ell} \left | Y_{t-t_\nu}^{n,\varepsilon} \right | \geq m \right \}
              \overset{a.s.}{\longrightarrow} \infty
            \end{equation}
            as $m\to\infty$, uniformly in $n$, $0<\varepsilon<1$ and $\ell\in\mathbb{N}_0$.
    \end{enumerate}
  \end{proposition}

  \begin{proof}
    It follows by Gronwall's inequality that
    \begin{equation}
      \sup_{t\in[0,1]} \left | X_t^{\varepsilon} - x_t \right |
      \leq e^{C} \varepsilon \sup_{t\in[0,1]} \left | L_t \right |,
    \end{equation}
    where $C$ is the common Lipschitz constant from \ref{asmp: 2}.
    Since $\left|\lceil n(t-t_\nu)\rceil/n-t\right|\leq\frac{\ell+1}{n}$ for all $t\in(t_{(\ell-1)\vee0},1]$,
    we have
    \begin{equation}
      \left | Y_{t-t_\nu}^{n,\varepsilon} - x_t \right |
      \leq e^C \varepsilon \sup_{s\in[0,1]} \left | L_s \right |
      + \sup_{\substack{\left | s-u \right | \leq (\ell+1)/n \\ s,u\in[0,1]}} \left | x_s - x_u \right |
    \end{equation}
    for all $t\in\left(t_{(\ell-1)\vee0},1\right]$.
    This implies \ref{3.1.i}.
    Moreover, \ref{3.1.ii} is immediate from the inequality
    \begin{equation}
      \left | Y_{t-t_\nu}^{n,\varepsilon} \right |
      \leq \sup_{s\in[0,1]} \left | x_s \right | + e^C \varepsilon \sup_{s\in[0,1]} \left | L_s \right |
      + \sup_{\substack{\left | s-u \right | \leq (\ell+1)/n \\ s,u\in[0,1]}} \left | x_s - x_u \right |
    \end{equation}
    for all $t\in\left(t_{(\ell-1)\vee0},1\right]$.
  \end{proof}

  \subsection{Inequalities for deterministic convergence}

    In this section, we prepare some inequalities for the solution of \eqref{eq: main ODE}.

    \begin{lemma}\label{lem: sup norm of f}
      Let $f$ be a function in $C^{\infty,0}(\mathbb{R}^d\times\Theta;\mathbb{R}^d)$ such that
      $\left \| f \right \|_{C^{\infty,0}(B_M\times\Theta)}< \infty$ for all $M>0$,
      and suppose the assumption \ref{asmp: 4}.
      Then,
      \begin{equation}
        \sup_{t\in[0,1]} \left | \frac{d^\ell}{dt^\ell} \left ( f(x_t,\theta) \right ) \right |
        \leq \ell! d^\ell \left \| b(\theta_0) \right \|_{C^\infty(B_M)}^\ell \left \| f \right \|_{C^{\infty,0}(B_M\times\Theta)}
      \end{equation}
      for all $\ell\in\mathbb{N}$.
    \end{lemma}

    \begin{proof}
      It is shown by induction that
      \begin{equation}
        \frac{d^\ell}{dt^\ell} \left ( f(x_t, \theta) \right )
        = \sum_{j_1=1}^d \dots \sum_{j_\ell=1}^d \sum_{|\alpha|+|\nu|=\ell} c_{\alpha,\nu} \left ( D_x^{\alpha_1} b_{j_1} \cdots D_x^{\alpha_\ell} b_{j_\ell} D_x^\nu f_{\theta} \right )_{x=x_t}, \quad
        \sum_{|\alpha|+|\nu|=\ell} c_{\alpha,\nu} = \ell!,
      \end{equation}
      where $\alpha=(\alpha_1,\dots,\alpha_\ell)$ for $\alpha_j\in\mathbb{N}_0^d$, $\nu\in\mathbb{N}_0^d$, $c_{\alpha,\nu}\in\mathbb{N}_0$. We write $b_i(x,\theta_0)$ and $f(x,\theta)$ simply as $b_i$ and $f_{\theta}$, respectively.
      Indeed, the derivative of each term with respect to $t$ is
      \begin{equation}
        c_{\alpha,\nu} \sum_{j_{\ell+1}=1}^d \left ( b_{j_{\ell+1}} D_x^{e_{j_{\ell+1}}} \left ( D_x^{\alpha_1} b_{j_1} \cdots D_x^{\alpha_\ell} b_{j_\ell} D_x^\nu f_{\theta} \right ) \right )_{x=x_t},
      \end{equation}
      where $e_{j}$ denotes $d$-dimensional multi-index with entry 1 at the $j$th coordinate, and entry zero elsewhere.
    \end{proof}

    \begin{lemma}\label{lem: al f deterministic}
      Let $f$ be a function as in Lemma \ref{lem: sup norm of f}.
      Under the assumption \ref{asmp: 4}, it follows that
      \begin{equation}
        \left | A_{\ell} f(\pmb{x}_{t_k:t_{k-\ell}},\theta) - \dashint_{t_{k-1}}^{t_k} f(x_s, \theta) \, ds \right |
        \leq \ell!n^{-(\ell+1)} d^{\ell+1} \left \| b(\theta_0) \right \|_{C^{\infty}(B_M)}^{\ell+1} \left \| f \right \|_{C^{\infty,0}(B_M\times\Theta)},
      \end{equation}
      where $k=\ell\vee1,\dots,n$, and $M=\sup_{t\in[0,1]}|x_t|$.
    \end{lemma}

    \begin{remark}
      When we employ $\tilde{A}_{\ell}f(\pmb{x}_{t_{k-1}:t_{k-\ell}},\theta)$ given by \eqref{eq: adams-bashforth sum} with \eqref{eq: def of adams-bashforth} as the version of the Adams-Bashforth method instead of $A_{\ell}f(\pmb{x}_{t_{k}:t_{k-\ell}},\theta)$,
      we obtain the following inequality:
      \begin{equation}
        \left | \tilde{A}_{\ell} f(\pmb{x}_{t_{k-1}:t_{k-\ell}},\theta) - \dashint_{t_{k-1}}^{t_k}f(x_s, \theta) \, ds \right |
        \leq \ell!n^{-\ell} d^{\ell} \left \| b(\theta_0) \right \|_{C^{\infty}(B_M)}^{\ell} \left \| f \right \|_{C^{\infty,0}(B_M\times\Theta)}.
      \end{equation}
    \end{remark}

    \begin{proof}
      It follows from \eqref{eq: polynomial} and \eqref{eq: adams-moulton sum} that
      \begin{equation}
        \left | A_{\ell} f(\pmb{x}_{t_k:t_{k-\ell}},\theta) - \dashint_{t_{k-1}}^{t_k}f(x_s, \theta) \, ds \right |
        = \left | \dashint_{t_{k-1}}^{t_k} \left ( P(s;f(x_\cdot,\theta),t_k,\dots,t_{k-\ell})  - f(x_s, \theta) \right ) \, ds \right |,
      \end{equation}
      where $k=\ell\vee1,\dots,n$, and $s\mapsto P(s;f(x_\cdot,\theta),t_k,\dots,t_{k-\ell})$ is the Lagrange interpolating polynomial through the points $(s,f(x_s,\theta))$, $s=t_k,\dots,t_{k-\ell}$.
      It holds from Theorem 3.1.1 in Davis \cite{Davis1975} that for each $s\in[t_{k-1},t_k]$ there exists $\xi_s\in(t_{k-\ell},t_k)$ such that
      \begin{equation}
        P(s;f(x_\cdot,\theta),t_k,\dots,t_{k-\ell})-f(x_s,\theta)
        = \frac{1}{(\ell+1)!} \left ( \frac{d^{\ell+1}}{dt^{\ell+1}} \left ( f(x_t,\theta) \right ) \right )_{t=\xi_s} \prod_{\nu=0}^{\ell} (s-t_{k-\nu}),
      \end{equation}
      and that
      \begin{equation}
        \begin{aligned}
          \dashint_{t_{k-1}}^{t_k} \left | \left ( \frac{d^{\ell+1}}{dt^{\ell+1}} \left ( f(x_t,\theta) \right ) \right )_{t=\xi_s} \prod_{\nu=0}^{\ell} (s-t_{k-\nu}) \right | \, ds
          &\leq \sup_{t\in[t_{k-\ell},t_k]} \left | \frac{d^{\ell+1}}{dt^{\ell+1}} \left ( f(x_t,\theta) \right ) \right | \dashint_{t_{k-1}}^{t_k} \prod_{\nu=0}^{\ell} \left | s-t_{k-\nu} \right | \, ds\\
          &\leq \frac{\ell!}{n^(\ell+1)} \sup_{t\in[0,1]} \left | \frac{d^{\ell+1}}{dt^{\ell+1}} \left ( f(x_t,\theta) \right ) \right |.
        \end{aligned}
      \end{equation}
      This yields the consequence.
    \end{proof}

  \subsection{Convergence theorems}

    \begin{proposition}\label{prop: a.s. conv sum | Al f(X) |^q}
      Let $f$ be a function as in Lemma \ref{lem: sup norm of f}.
      Suppose the assumptions \ref{asmp: 2} and \ref{asmp: 4},
      and that the family $\{f(\cdot,\theta)\}_{\theta\in\Theta}$ is equi-Lipschitz continuous.
      If $\ell/n\to0$ and $2^\ell\varepsilon\to0$ as $n\to\infty$ and $\varepsilon\to0$, then for all $q\geq1$
      \begin{equation}
        \frac{1}{n} \sum_{k=\ell\vee1}^n \left | A_\ell f(\pmb{X}_{t_k:t_{k-\ell}}^{\varepsilon}, \theta) \right |^q
        \overset{a.s.}{\longrightarrow}
        \int_0^1 \left | f(x_t,\theta) \right |^q \, dt
      \end{equation}
      as $n\to\infty$ and $\varepsilon\to0$, uniformly in $\theta\in\Theta$.
    \end{proposition}

    \begin{proof}
      We use the triangle inequality to obtain that
      \begin{multline}
        \left | \left ( \frac{1}{n} \sum_{k=\ell\vee1}^n \left | A_\ell f(\pmb{X}_{t_k:t_{k-\ell}}^{\varepsilon}, \theta) \right |^q \right )^{1/q}
        - \left ( \int_0^1 \left | f(x_t,\theta) \right |^q \, dt \right )^{1/q} \right | \\
        \leq
        \left ( \frac{1}{n} \sum_{k=\ell\vee1}^n \left | A_\ell f(\pmb{X}_{t_k:t_{k-\ell}}^{\varepsilon}, \theta) - A_\ell f(\pmb{x}_{t_k:t_{k-\ell}}, \theta) \right |^q \right )^{1/q} \\
         + \left ( \frac{1}{n} \sum_{k=\ell\vee1}^n \left | A_\ell f(\pmb{x}_{t_k:t_{k-\ell}}, \theta) - \dashint_{t_{k-1}}^{t_k} f(x_t,\theta) \, dt \right |^q \right )^{1/q} \\
         + \left | \left ( \frac{1}{n} \sum_{k=\ell\vee1}^n \left | \dashint_{t_{k-1}}^{t_k} f(x_t,\theta) \, dt \right |^q \right )^{1/q} - \left ( \int_0^1 \left | f(x_t,\theta) \right |^q \, dt \right )^{1/q} \right |.
      \end{multline}
      The second and the third term in the right-hand side converge to zero as $n\to\infty$ and $\ell/n\to0$, uniformly in $\theta\in\Theta$, by Lemma \ref{lem: al f deterministic} and Lemma \ref{lem: ave int conv}.
      From Lemma \ref{lem: L1-ineq for adams coef}, the first term is estimated from above by
      \begin{equation}
        \left ( \sum_{\nu=0}^{\ell} \left | \beta_{\ell\nu} \right | \right )
        \sup_{s\in[0,1]} \left | f(X_s^\varepsilon,\theta) - f(x_s,\theta) \right |
        \leq C 2^\ell
        \sup_{s\in[0,1]} \left | X_s^\varepsilon - x_s \right |,
      \end{equation}
      where $C$ is the common Lipschitz constant for $f$.
      This converges almost surely to zero as $2^\ell\varepsilon\to0$, uniformly in $\theta\in\Theta$, as we saw in the proof of Proposition \ref{prop: Yt to Xt}.
    \end{proof}

    \begin{remark}
      If we employ $\tilde{A}_{\ell}f(\pmb{X}_{t_{k-1}:t_{k-\ell}}^\varepsilon,\theta)$ instead of $A_{\ell}f(\pmb{x}_{t_{k}:t_{k-\ell}}^\varepsilon,\theta)$,
      the convergence in Proposition \ref{prop: a.s. conv sum | Al f(X) |^q} holds under $\ell2^{\ell}\varepsilon\to0$.
    \end{remark}

    \begin{remark}
      It is easy to check that, for $f$ and $g$ satisfying the same assumptions as in Proposition \ref{prop: a.s. conv sum | Al f(X) |^q},
      \begin{equation}
        \frac{1}{n} \sum_{k=\ell\vee1}^n A_\ell f(\pmb{X}_{t_k:t_{k-\ell}}^{\varepsilon}, \theta) \cdot A_\ell g(\pmb{X}_{t_k:t_{k-\ell}}^{\varepsilon}, \theta)
        \overset{a.s.}{\longrightarrow}
        \int_0^1 f(x_t,\theta) \cdot g(x_t,\theta) \, dt.
      \end{equation}
      This convergence will appear in the proof of Proposition \ref{prop: 3.3} (ii).
      We can also say that
      \begin{equation}
        \frac{1}{n} \sum_{k=\ell\vee1}^n A_\ell f(\pmb{X}_{t_k:t_{k-\ell}}^{\varepsilon}, \theta)
        \overset{a.s.}{\longrightarrow}
        \int_0^1 f(x_t,\theta) \, dt,
      \end{equation}
      though we will not need this in the paper.
    \end{remark}

    \begin{lemma}\label{lem: conv 4}
      Let $f$ be a function as in Proposition \ref{prop: a.s. conv sum | Al f(X) |^q}.
      Suppose the assumptions \ref{asmp: 2} and \ref{asmp: 4},
      and that 
      $f$ is differentiable with respect to $\theta\in\Theta_0$,
      and the families $\left \{ \partial_{\theta_j}f(\cdot,\theta) \right \}_{\theta\in\Theta_0}$ $(j=1,\dots,p)$ are equi-Lipschitz continuous.
      If $\ell2^\ell/n\to0$ and $2^\ell\varepsilon\to0$ as $n\to\infty$ and $\varepsilon\to0$,
      then it holds that
      \begin{equation}
        \sum_{k=\ell\vee1}^n A_\ell f(\pmb{X}_{t_k:t_{k-\ell}}^{\varepsilon}, \theta) \cdot \left ( L_{t_k} - L_{t_{k-1}} \right )
        \overset{P_{\theta_0}}{\longrightarrow}
        \int_0^1 f(x_t, \theta) \cdot dL_t
      \end{equation}
      as $n\to\infty$ and $\varepsilon\to0$, uniformly in $\theta\in\Theta$.
    \end{lemma}

    \begin{proof}
      Since
      \begin{equation}
        \sum_{k=\ell\vee1}^n A_\ell f(\pmb{X}_{t_k:t_{k-\ell}}^{\varepsilon}, \theta) \cdot \left ( L_{t_k} - L_{t_{k-1}} \right )
        = \sum_{\nu=0}^{\ell} \beta_{\ell\nu} \int_{t_{(\ell-1)\vee0}}^1 f ( Y_{t-t_{\nu}}^{n,\varepsilon}, \theta ) \cdot dL_t
      \end{equation}
      and $\sum_{\nu=0}^{\ell} \beta_{\ell\nu} = 1$, we have
      \begin{multline}
        \sum_{k=\ell\vee1}^n A_\ell f(\pmb{X}_{t_k:t_{k-\ell}}^{\varepsilon}, \theta) \cdot \left ( L_{t_k} - L_{t_{k-1}} \right )
        - \int_0^1 f(x_t, \theta) \cdot dL_t \\
        = \sum_{\nu=0}^{\ell} \beta_{\ell\nu} \int_{t_{(\ell-1)\vee0}}^1 \left ( f ( Y_{t-t_{\nu}}^{n,\varepsilon}, \theta ) - f(x_t, \theta) \right ) \cdot dL_t
        - \int_0^{t_{(\ell-1)\vee0}} f(x_t, \theta) \cdot dL_t.
      \end{multline}
      The last term converges almost surely to zero as $n\to\infty$ and $\ell/n\to0$, uniformly in $\theta\in\Theta$.
      Let us denote
      \begin{equation}
        \tilde{L}_t = \sigma B_t + \int_0^t \int_{|z|\leq1} z \, \tilde{N}(ds,dz),
      \end{equation}
      then $L_t = a t + \tilde{L}_t + \int_0^t \int_{|z|>1} z \, N(ds,dz)$.
      We have
      \begin{multline}
        \left | \sum_{\nu=0}^{\ell} \beta_{\ell\nu} \int_{t_{(\ell-1)\vee0}}^1 \int_{|z|>1} \left ( f ( Y_{t-t_{\nu}}^{n,\varepsilon}, \theta ) - f(x_t, \theta) \right ) \cdot z \, N(dt,dz) \right | \\
        \leq \left ( \sum_{\nu=0}^{\ell} \left | \beta_{\ell\nu} \right | \right ) \sup_{\nu=0,\dots,\ell} \int_{t_{(\ell-1)\vee0}}^1 \int_{|z|>1} \left | f ( Y_{t-t_{\nu}}^{n,\varepsilon}, \theta ) - f(x_t, \theta) \right | |z| \, N(dt,dz) \\
        \leq C 2^\ell \sup_{\substack{\nu=0,\dots,\ell \\ t\in[t_{(\ell-1)\vee0},1]}} \left | Y_{t-t_{\nu}}^{n,\varepsilon} - x_t \right | \int_{0}^1 \int_{|z|>1} |z| \, N(dt,dz),
      \end{multline}
      which converges almost surely to zero as $n\to\infty$, $\varepsilon\to0$, $\ell2^\ell/n\to0$ and $2^\ell\varepsilon\to0$, uniformly in $\theta\in\Theta$, by Proposition \ref{prop: Yt to Xt}.
      Analogously, we obtain
      \begin{equation}
        \left | \sum_{\nu=0}^{\ell} \beta_{\ell\nu} \int_{t_{(\ell-1)\vee0}}^1 \left ( f ( Y_{t-t_{\nu}}^{n,\varepsilon}, \theta ) - f(x_t, \theta) \right ) \cdot a \, dt \right |
        \overset{a.s.}{\longrightarrow} 0
      \end{equation}
      as $n\to\infty$, $\varepsilon\to0$, $\ell2^\ell/n\to0$ and $2^\ell\varepsilon\to0$, uniformly in $\theta\in\Theta$.

      Analogous to the proof of Lemma 4 in Ogihara and Yoshida \cite{ogihara2011quasi}, it follows from Markov's inequality and Morrey's inequality (see, \textit{e.g.}, Theorem 5 in Evans \cite[Section 5.6]{Evans2010partial}) that for any $q\in(p,\infty]$ and $\eta>0$
      \begin{multline}\label{eq: P(sup |int dLt| > eta)}
        P \left ( \sup_{\theta\in\Theta} \left | \sum_{\nu=0}^{\ell} \beta_{\ell\nu} \int_{t_{(\ell-1)\vee0}}^1 \pmb{1}_{\{t\leq\tau_m^{n,\varepsilon,\ell}\}} \left ( f ( Y_{t-t_{\nu}}^{n,\varepsilon}, \theta ) - f(x_t, \theta) \right )  \cdot d\tilde{L}_t \right | > \eta \right ) \\
        \leq \frac{1}{\eta} E \left [ \sup_{\theta\in\Theta} \left | \int_{t_{(\ell-1)\vee0}}^1 \pmb{1}_{\{t\leq\tau_m^{n,\varepsilon,\ell}\}} \sum_{\nu=0}^{\ell} \beta_{\ell\nu} \left ( f ( Y_{t-t_{\nu}}^{n,\varepsilon}, \theta ) - f(x_t, \theta) \right ) \cdot d\tilde{L}_t \right | \right ] \\
        \leq \frac{C}{\eta} E \left [ \left \| \int_{t_{(\ell-1)\vee0}}^1 \pmb{1}_{\{t\leq\tau_m^{n,\varepsilon,\ell}\}} \sum_{\nu=0}^{\ell} \beta_{\ell\nu} \left ( f(Y_{t-t_{\nu}}^{n,\varepsilon}, \cdot\,) - f(x_t, \cdot\,) \right ) \cdot d\tilde{L}_t \right \|_{W^{1,q}(\Theta)} \right ],
      \end{multline}
      where $C$ is a constant depending only on $p,q$ and $\Theta$.
      It follows from H\"{o}lder's inequality and Fubini's theorem that
      \begin{multline}\label{eq: E int dLt to 0 (1)}
        P \left ( \sup_{\theta\in\Theta} \left | \sum_{\nu=0}^{\ell} \beta_{\ell\nu} \int_{t_{(\ell-1)\vee0}}^1 \pmb{1}_{\{t\leq\tau_m^{n,\varepsilon,\ell}\}} \left ( f ( Y_{t-t_{\nu}}^{n,\varepsilon}, \theta ) - f(x_t, \theta) \right )  \cdot d\tilde{L}_t \right | > \eta \right ) \\
        \leq \frac{C}{\eta} \left ( \int_\Theta E \left [ \left | \int_{t_{(\ell-1)\vee0}}^1 \pmb{1}_{\{t\leq\tau_m^{n,\varepsilon,\ell}\}} \sum_{\nu=0}^{\ell} \beta_{\ell\nu} \left ( f(Y_{t-t_{\nu}}^{n,\varepsilon}, \theta) - f(x_t, \theta) \right ) \cdot d\tilde{L}_t \right |^q \right ] d\theta \right )^{1/q} \\
        +\frac{C}{\eta} \left ( \int_\Theta E \left [ \left | \int_{t_{(\ell-1)\vee0}}^1 \pmb{1}_{\{t\leq\tau_m^{n,\varepsilon,\ell}\}} \sum_{\nu=0}^{\ell} \beta_{\ell\nu} \left ( \partial_{\theta_j} f (Y_{t-t_{\nu}}^{n,\varepsilon}, \theta) - \partial_{\theta_j} f (x_t, \theta) \right ) \cdot d\tilde{L}_t \right |^q \right ] d\theta \right )^{1/q}
      \end{multline}
      for $j=1,\dots,p$.
      By the moment inequality for stochastic integrals (see, \textit{e.g.}, Theorem 7.1 in Chapter 1 in Mao \cite{mao2008stochastic}), for $q\geq2$ we obtain
      \begin{multline}
        \int_\Theta E \left [ \left | \int_{t_{(\ell-1)\vee0}}^1 \pmb{1}_{\{t\leq\tau_m^{n,\varepsilon,\ell}\}} \sum_{\nu=0}^{\ell} \beta_{\ell\nu} \left ( f(Y_{t-t_{\nu}}^{n,\varepsilon}, \theta) - f(x_t, \theta) \right ) \cdot dB_t \right |^q \right ] d\theta \\
        \leq \left ( \frac{q(q-1)}{2} \right )^{q/2} \int_\Theta E \left [ \int_{t_{(\ell-1)\vee0}}^1 \pmb{1}_{\{t\leq\tau_m^{n,\varepsilon,\ell}\}} \left | \sum_{\nu=0}^{\ell} \beta_{\ell\nu} \left ( f(Y_{t-t_{\nu}}^{n,\varepsilon}, \theta) - f(x_t, \theta) \right ) \right |^q dt \right ] d\theta \\
        \leq \left ( \frac{q(q-1)}{2} \right )^{q/2} C^q \int_\Theta E \left [ \int_{t_{(\ell-1)\vee0}}^1 \pmb{1}_{\{t\leq\tau_m^{n,\varepsilon,\ell}\}} 2^{q\ell} \sup_{\nu=0,\dots,\ell} \left | Y_{t-t_{\nu}}^{n,\varepsilon} - x_t \right |^q dt \right ] d\theta,
      \end{multline}
      and by Kunita's inequality (see, \textit{e.g.}, Theorem 4.4.23 in Applebaum \cite{applebaum2009levy}),
      for $q\geq2$, there exists $D(q)>0$ such that
      \begin{align}
        &\int_\Theta E \left [ \left | \int_{t_{(\ell-1)\vee0}}^1 \int_{0<|z|\leq1} \pmb{1}_{\{t\leq\tau_m^{n,\varepsilon,\ell}\}} \sum_{\nu=0}^{\ell} \beta_{\ell\nu} \left ( f(Y_{t-t_{\nu}}^{n,\varepsilon}, \theta) - f(x_t, \theta) \right ) \cdot z \, \tilde{N}(ds,dz) \right |^q \right ] d\theta \\
        \leq&\ D(q) \int_\Theta \left \{ E \left [ \left ( \int_{t_{(\ell-1)\vee0}}^1 \int_{0<|z|\leq1} \pmb{1}_{\{t\leq\tau_m^{n,\varepsilon,\ell}\}}  \left | \sum_{\nu=0}^{\ell} \beta_{\ell\nu} \left ( f(Y_{t-t_{\nu}}^{n,\varepsilon}, \theta) - f(x_t, \theta) \right ) \cdot z \right |^2 \nu(dz) \, dt \right )^{q/2} \right ] \right . \\
        &+ \left . E \left [ \int_{t_{(\ell-1)\vee0}}^1 \int_{0<|z|\leq1} \pmb{1}_{\{t\leq\tau_m^{n,\varepsilon,\ell}\}} \left | \sum_{\nu=0}^{\ell} \beta_{\ell\nu} \left ( f(Y_{t-t_{\nu}}^{n,\varepsilon}, \theta) - f(x_t, \theta) \right ) \cdot z \right |^q \nu(dz) \, dt \right ] \right \} \\
        \leq&\ D(q) C^q \left \{ \left ( \int_{0<|z|\leq1} |z|^2 \, \nu(dz) \right )^{q/2}
        E \left [ \left ( \int_{t_{(\ell-1)\vee0}}^1 \pmb{1}_{\{t\leq\tau_m^{n,\varepsilon,\ell}\}} \sup_{\nu=0,\dots,\ell} 2^{2\ell} \left | Y_{t-t_\nu}^{n,\varepsilon} - x_t \right |^2 dt \right )^{q/2} \right ] \right . \\
        &\left . + \left ( \int_{0<|z|\leq1} |z|^q \, \nu(dz) \right )
        E \left [ \int_{t_{(\ell-1)\vee0}}^1 \pmb{1}_{\{t\leq\tau_m^{n,\varepsilon,\ell}\}} \sup_{\nu=0,\dots,\ell} 2^{q\ell} \left | Y_{t-t_\nu}^{n,\varepsilon} - x_t \right |^q dt \right ] \right \}.
      \end{align}
      Both converge to zero as $\ell2^{\ell}/n\to0$ and $2^\ell\varepsilon\to0$, by dominated convergence theorem, and so does \eqref{eq: P(sup |int dLt| > eta)}.
    \end{proof}

    \begin{proposition}\label{prop: 2.7}
      Let $f$ be a function as in Lemma \ref{lem: conv 4}.
      Under the assumptions \ref{asmp: 2} and \ref{asmp: 4},
      if $\ell2^{4\ell}/n\to0$, $2^\ell\varepsilon\to0$ and $\ell2^{2\ell}/n\varepsilon\to0$ as $n\to\infty$ and $\varepsilon\to0$,
      then it holds that
      \begin{equation}
        \frac{1}{\varepsilon} \sum_{k=\ell\vee1}^n
        A_\ell f(\pmb{X}_{t_k:t_{k-\ell}}^{\varepsilon}, \theta)
        \cdot \left ( X_{t_k}^{\varepsilon} - X_{t_{k-1}}^{\varepsilon} - \frac{1}{n} A_\ell b(\pmb{X}_{t_k:t_{k-\ell}}^{\varepsilon}, \theta_0) \right )
        \overset{P_{\theta_0}}{\longrightarrow}
        \int_0^1 f(x_t,\theta) \cdot dL_t
      \end{equation}
      as $n\to\infty$ and $\varepsilon\to0$, uniformly in $\theta\in\Theta$.
    \end{proposition}

    \begin{remark}
      This lemma will be essentially used for the case $\theta=\theta_0$.
    \end{remark}

    \begin{proof}
      It follows that
      \begin{align}
        &\frac{1}{\varepsilon} \sum_{k=\ell\vee1}^n A_\ell f(\pmb{X}_{t_k:t_{k-\ell}}^{\varepsilon}, \theta) \cdot \left ( X_{t_k}^{\varepsilon} - X_{t_{k-1}}^{\varepsilon} - \frac{1}{n} A_\ell b(\pmb{X}_{t_k:t_{k-\ell}}^{\varepsilon}, \theta_0) \right ) \\
        &\qquad\begin{aligned}
          &= \frac{1}{\varepsilon} \sum_{k=\ell\vee1}^n A_\ell f(\pmb{X}_{t_k:t_{k-\ell}}^{\varepsilon}, \theta) \cdot \left ( \int_{t_{k-1}}^{t_k} b(X_t^{\varepsilon},\theta_0) \, dt - \frac{1}{n} A_\ell b(\pmb{X}_{t_k:t_{k-\ell}}^{\varepsilon}, \theta_0) \right ) \\
          &\quad + \sum_{k=\ell\vee1}^n A_\ell f(\pmb{X}_{t_k:t_{k-\ell}}^{\varepsilon}, \theta) \cdot \left ( L_{t_k} - L_{t_{k-1}} \right ) \\
          &=: J_1 + J_2.
        \end{aligned}
        \label{eq: int f dL 1}
      \end{align}
      From Lemma \ref{lem: conv 4}, $J_2$ converges to $\int_0^1 f(x_t,\theta) \cdot dL_t$ in $P_{\theta_0}$ as $n\to\infty$, $\varepsilon\to0$, $\ell2^\ell/n\to0$ and $2^\ell\varepsilon\to0$, uniformly in $\theta\in\Theta$, and
      \begin{multline}
        |J_1|=
        \frac{1}{\varepsilon} \left | \sum_{\nu,\mu=0}^{\ell} \beta_{\ell\nu} \beta_{\ell\mu}
        \sum_{k=\ell\vee1}^n \int_{t_{k-1}}^{t_k} f(X_{t_{k-\mu}}^\varepsilon, \theta) \cdot \left ( b(X_t^\varepsilon, \theta_0) - b(X_{t_{k-\nu}}^{\varepsilon}, \theta_0) \right ) dt \right | \\
        \leq \left ( \sum_{\nu=0}^{\ell} \left | \beta_{\ell\nu} \right | \right )^2
        \frac{1}{n\varepsilon} \sum_{k=\ell\vee1}^n \sup_{\substack{\nu,\mu=0,\dots,\ell\\t\in[t_{k-1},t_k]}} \left | f(X_{t_{k-\mu}}^\varepsilon, \theta) \cdot \left ( b ( X_t^\varepsilon, \theta_0 ) - b(X_{t_{k-\nu}}^{\varepsilon}, \theta_0) \right ) \right | \\
        \leq \frac{C2^{2\ell}}{n\varepsilon} \sum_{k=\ell\vee1}^n \sup_{\substack{\nu,\mu=0,\dots,\ell\\t\in[t_{k-1},t_k]}} \left | f(X_{t_{k-\mu}}^\varepsilon, \theta) \right | \left | X_t^\varepsilon - X_{t_{k-\nu}}^{\varepsilon} \right |,
      \end{multline}
      where $C$ is a Lipschitz constant in \ref{asmp: 2}.
      For $t\in[t_{(\ell-1)\vee0},1)$,
      \begin{equation}
        \begin{aligned}
          \left | X_t^\varepsilon - X_{t_{k-\nu}}^\varepsilon \right |
          &= \left | \int_{t_{k-\nu}}^t b(X_s^\varepsilon,\theta_0) \, ds + \varepsilon (L_t - L_{t_{k-\nu}}) \right | \\
          &\leq C \int_{t_{k-\nu}}^t \left | X_s^\varepsilon -X_{t_{k-\nu}}^\varepsilon \right | \, ds + \frac{\ell}{n} \left | b(X_{t_{k-\nu}}^\varepsilon,\theta_0) \right |
          + \varepsilon \sup_{s\in[t_{k-1},t_k]} \left | L_s - L_{t_{k-\nu}} \right |
        \end{aligned}
      \end{equation}
      and by Gronwall's inequality, we obtain
      \begin{equation}
        \left | X_t^\varepsilon - X_{t_{k-\nu}}^\varepsilon \right |
        \leq e^{C(t-t_{k-\nu})} \left ( \frac{\ell}{n} \left | b(X_{t_{k-\nu}}^\varepsilon,\theta_0) \right |
        + \varepsilon \sup_{s\in[t_{k-1},t_k]} \left | L_s - L_{t_{k-\nu}} \right | \right ).
      \end{equation}
      Thus,
      \begin{equation}
        |J_1| \leq
        Ce^{C\ell/n} \left ( \frac{\ell2^{2\ell}}{n\varepsilon} \sup_{s,t\in[0,1]} \left | b(X_s^\varepsilon,\theta_0) f(X_t^\varepsilon,\theta) \right |
        + \frac{2^{2\ell}}{n} \sum_{k=\ell\vee1}^n \sup_{\substack{\nu,\mu=0,\dots,\ell\\s\in[t_{k-1},t_k]}} \left | f(X_{t_{k-\mu}}^\varepsilon, \theta) \right | \left | L_s - L_{t_{k-\nu}} \right | \right ).
      \end{equation}
      The next to the last term converges almost surely to zero as $\varepsilon\to0$ and $\ell2^{2\ell}/n\varepsilon\to0$, uniformly in $\theta\in\Theta$.
      We remain to prove that
      \begin{equation}
        \frac{2^{2\ell}}{n} \sum_{k=\ell\vee1}^n \sup_{\substack{\nu=0,\dots,\ell \\ s\in[t_{k-1},t_k]}} \left | L_s - L_{t_{k-\nu}} \right |
        \overset{P_{\theta_0}}{\longrightarrow} 0.
      \end{equation}
      This follows from the fact that
      \begin{equation}
        \sup_{\substack{\nu=0,\dots,\ell \\ s\in[t_{k-1},t_k]}} \left | L_s - L_{t_{k-\nu}} \right |
        \leq \sup_{\substack{\nu=0,\dots,\ell \\ s\in[t_{k-1},t_k]}} \left | \tilde{L}_s - \tilde{L}_{t_{k-\nu}} \right |
        + (t_k - t_{k-\ell\vee1}) + \int_{t_{k-\ell\vee1}}^{t_k} \int_{|z|>1} |z| \, N(ds,dz),
      \end{equation}
      where
      \begin{equation}
        \begin{aligned}
          & \frac{2^{2\ell}}{n} \sum_{k=\ell\vee1}^n (t_k - t_{k-\ell\vee1})
          = \frac{2^{2\ell}}{n} \sum_{\nu=0}^{(\ell-1)\vee0} (t_{n-\nu} - t_\nu)
          = \frac{2^{2\ell}}{n} \sum_{\nu=0}^{(\ell-1)\vee0} \left ( 1 - \frac{2\nu}{n} \right ) \leq \frac{\ell2^{2\ell}}{n} \to 0, \\
          & \frac{2^{2\ell}}{n} \sum_{k=\ell\vee1}^n \int_{t_{k-\ell\vee1}}^{t_k} \int_{|z|>1} |z| \, N(ds,dz)
          \leq \frac{\ell2^{2\ell}}{n} \int_{0}^{1} \int_{|z|>1} |z| \, N(ds,dz)
          \overset{a.s.}{\longrightarrow} 0
          \qquad \text{as} \quad\frac{\ell2^{2\ell}}{n} \to 0,
        \end{aligned}
      \end{equation}
      and by Doob's martingale inequality (see, \textit{e.g.}, Theorem 2.1.5 in Applebaum \cite{applebaum2009levy})
      \begin{equation}
        \begin{aligned}
          \frac{2^{2\ell}}{n} \sum_{k=\ell\vee1}^n
          E \left [ \sup_{\substack{\nu=0,\dots,\ell \\ s\in[t_{k-1},t_k]}} \left | \tilde{L}_s - \tilde{L}_{t_{k-\nu}} \right | \right ]
          &\leq 2^{2\ell} \left ( \frac{1}{n} \sum_{k=\ell\vee1}^n E \left [ \sup_{\substack{\nu=0,\dots,\ell \\ s\in[t_{k-1},t_k]}} \left | \tilde{L}_s - \tilde{L}_{t_{k-\nu}} \right |^2 \right ] \right )^{1/2} \\
          &\leq C \left ( \frac{2^{4\ell}}{n} \sum_{k=\ell\vee1}^n E \left [ \left | \tilde{L}_{t_k} - \tilde{L}_{t_{k-\nu}} \right |^2 \right ] \right )^{1/2} \\
          &\leq C \left ( \frac{\ell2^{4\ell}}{n} \left ( \|\sigma\|_F^2 + \int_{|z|\leq1} |z|^2 \, \nu(dz) \right ) \right )^{1/2} \to 0
          \quad \text{as}~ \frac{\ell2^{4\ell}}{n} \to 0
        \end{aligned}
      \end{equation}
      with some positive constant $C$ independent of $n,\varepsilon,\ell$.
      Thus, for any $\eta>0$,
      \begin{multline}
        P \left ( \sup_{\theta\in\Theta} \frac{2^{2\ell}}{n} \sum_{k=\ell\vee1}^n \sup_{\substack{\nu,\mu=0,\dots,\ell\\t\in[t_{k-1},t_k]}} \left | f(X_{t_{k-\mu}}^\varepsilon, \theta) \right | \left | L_t - L_{t_{k-\nu}} \right | > \eta \right ) \\
        \leq P (1>\tau_m^{n,\varepsilon,\ell})
        + P \left ( \left \| f \right \|_{C(B_m\times\Theta)} \frac{2^{2\ell}}{n} \sum_{k=\ell\vee1}^n \sup_{\substack{\nu=0,\dots,\ell\\t\in[t_{k-1},t_k]}} \left | L_t - L_{t_{k-\nu}} \right | > \eta \right )
      \end{multline}
      converges in $P_{\theta_0}$ to zero
      as $n\to\infty$, $\varepsilon\to0$ and $\ell2^{4\ell}/n\to0$.
    \end{proof}

    Analogously, we obtain the following proposition.

    \begin{proposition}\label{prop: 2.8}
      Let $f$ be a function as in Lemma \ref{lem: conv 4}.
      Under the assumptions \ref{asmp: 2} and \ref{asmp: 4},
      if $2^\ell\varepsilon\to0$ and $\ell2^{2\ell}/n$ is bounded as $n\to\infty$ and $\varepsilon\to0$,
      then it holds that
      \begin{equation}
        \sum_{k=\ell\vee1}^n
        A_\ell f(\pmb{X}_{t_k:t_{k-\ell}}^{\varepsilon}, \theta)
        \cdot
        \left ( X_{t_k}^{\varepsilon} - X_{t_{k-1}}^{\varepsilon} - \frac{1}{n} A_\ell b(\pmb{X}_{t_k:t_{k-\ell}}^{\varepsilon}, \theta_0) \right )
        \overset{P_{\theta_0}}{\longrightarrow}0
      \end{equation}
      as $n\to\infty$ and $\varepsilon\to0$, uniformly in $\theta\in\Theta$.
    \end{proposition}

\section{Main result}
  To prove our main results, we essentially follow the idea by Uchida \cite{uchida2004estimation} and Long \textit{et al.} \cite{long2013least}.

  \begin{theorem}[Consistency]\label{thm: consistency}
    Under conditions \ref{asmp: 2}-\ref{asmp: 5}, the least squares estimator $\hat{\theta}_{n,\varepsilon,\ell}$ given in \eqref{eq: def of estimator theta_n} is consistent to $\theta_0$, \textit{i.e.},
    if $2^\ell \varepsilon\to0$ and $\ell2^{2\ell}/n$ is bounded as $n\to\infty$ and $\varepsilon\to0$, then
    \begin{equation}
      \hat{\theta}_{n,\varepsilon,\ell}
      \overset{P_{\theta_0}}{\longrightarrow}
      \theta_0
    \end{equation}
    as $n\to\infty$ and $\varepsilon\to0$.
  \end{theorem}

  \begin{proof}
    Let $f(x,\theta)=b(x,\theta_0)-b(x,\theta)$.
    Since
    \begin{align}
      \Phi_{n,\varepsilon,\ell}(\theta)
      =&\ 2 \sum_{k=\ell\vee1}^n \left ( X_{t_k}^{\varepsilon} - X_{t_{k-1}}^{\varepsilon} - \frac{1}{n} A_\ell b(\pmb{X}_{t_k:t_{k-\ell}}^{\varepsilon}, \theta_0) \right )
      \cdot A_\ell f(\pmb{X}_{t_k:t_{k-\ell}}^{\varepsilon}, \theta) \\
      &+ \frac{1}{n} \sum_{k=\ell\vee1}^n \left | A_\ell f(\pmb{X}_{t_k:t_{k-\ell}}^{\varepsilon}, \theta) \right |^2,
    \end{align}
    it follows from Proposition \ref{prop: a.s. conv sum | Al f(X) |^q} and \ref{prop: 2.8} that
    for any $\eta>0$, if $2^\ell \varepsilon\to0$ and $\ell2^{2\ell}/n$ is bounded, then
    \begin{equation}
      P \left ( \sup_{\theta\in\Theta} \left | \Phi_{n,\varepsilon,\ell}(\theta) - \int_0^1 \left | f(x_s,\theta) \right |^2 ds \right | > \eta \right ) \to 0
    \end{equation}
    as $n\to\infty$, $\varepsilon\to0$.
    Also, \ref{asmp: 5} implies that for any $\delta>0$
    \begin{equation}
      \inf_{|\theta- \theta_0|>\delta} \int_0^1 \left | f(x_s,\theta) \right |^2 ds
      > \int_0^1 \left | f(x_s,\theta_0) \right |^2 ds = 0.
    \end{equation}
    Thus, it follows from Theorem 5.9 in van der Vaart \cite{vandervaart1998} that $\hat{\theta}_{n,\varepsilon,\ell}$ is consistent to $\theta_0$.
  \end{proof}

  \begin{theorem}[Asymptotic distribution]\label{thm: asymp dstr}
    Under conditions \ref{asmp: 2}-\ref{asmp: 5},
    if $\ell2^{4\ell}/n\to0$, $2^\ell\varepsilon\to0$ and $\ell2^{2\ell}/n\varepsilon\to0$ as $n\to\infty$ and $\varepsilon\to0$, then
    \begin{equation}
      \varepsilon^{-1} \left ( \hat{\theta}_{n,\varepsilon,\ell} - \theta_0 \right )
      \overset{P_{\theta_0}}{\longrightarrow}
      I(\theta_0)^{-1} S(\theta_0)
    \end{equation}
    as $n\to\infty$ and $\varepsilon\to0$,
    where $I(\theta)$ is a $p\times p$ positive definite symmetric matrix with the $(i,j)$-th entry
    \begin{equation}
      I^{ij} (\theta) := \int_0^1 \partial_{\theta_i} b (x_t,\theta) \cdot \partial_{\theta_j} b (x_t,\theta) \, dt,
    \end{equation}
    and $S(\theta)$ is a $p$-dimensional vector with the $i$-th entry
    \begin{equation}
      S_{i}(\theta) := \int_0^1 \partial_{\theta_i} b (x_t, \theta) \cdot dL_t
    \end{equation}
    for $\theta\in\Theta$, respectively.
  \end{theorem}

  \begin{remark}
    The consistency of $\tilde{\theta}_{n,\varepsilon,\ell}$ given by \eqref{eq: def of estimator theta_n via adams-bashforth} also holds
    if $\ell2^\ell\varepsilon\to0$ and $\ell^32^{2\ell}/n$ is bounded as $n\to\infty$ and $\varepsilon\to0$.
    In Theorem \ref{thm: asymp dstr},
    the corresponding convergence for $\tilde{\theta}_{n,\varepsilon,\ell}$ holds
    if $\ell^52^{4\ell}/n\to0$, $\ell2^\ell\varepsilon\to0$ and $\ell^32^{2\ell}/n\varepsilon\to0$ as $n\to\infty$ and $\varepsilon\to0$.
  \end{remark}

  To prove Theorem \ref{thm: asymp dstr}, we prepare the following proposition.

  \begin{proposition}\label{prop: 3.3}
    Assume the conditions \ref{asmp: 2}-\ref{asmp: 5}.
    \begin{enumerate}
      \item[\rm (i)] If $\ell2^{4\ell}/n\to0$, $2^\ell\varepsilon\to0$ and $\ell2^{2\ell}/n\varepsilon\to0$ as $n\to\infty$ and $\varepsilon\to0$, then
                  \begin{equation}
                    \varepsilon^{-1} \partial_{\theta_i} \Phi_{n,\varepsilon,\ell} (\theta_0)
                    \overset{P_{\theta_0}}{\longrightarrow}
                    -2 S_i(\theta_0)
                  \end{equation}
                  as $n\to\infty$ and $\varepsilon\to0$.
      \item[\rm (ii)] If $2^{\ell}\varepsilon\to0$ and $\ell2^{2\ell}/n$ is bounded as $n\to\infty$ and $\varepsilon\to0$, then
                  \begin{equation}\label{eq: conv to I(theta)}
                    \partial_{\theta_i} \partial_{\theta_j} \Phi_{n,\varepsilon,\ell}(\theta)
                    \overset{P_{\theta_0}}{\longrightarrow}
                    2 I^{ij}(\theta)
                  \end{equation}
                  as $n\to\infty$ and $\varepsilon\to0$, uniformly in $\theta\in\Theta$.
    \end{enumerate}
  \end{proposition}

  \begin{proof}
    i) Since we have
    \begin{equation}
      \partial_{\theta_i} \Phi_{n,\varepsilon,\ell}(\theta)
      = -2 \sum_{k=\ell\vee1}^n A_\ell \partial_{\theta_i} b(\pmb{X}_{t_k:t_{k-\ell}}^{\varepsilon}, \theta)
      \cdot \left ( X_{t_k}^{\varepsilon} - X_{t_{k-1}}^{\varepsilon} - \frac{1}{n} A_\ell b(\pmb{X}_{t_k:t_{k-\ell}}^{\varepsilon}, \theta_0) \right ),
    \end{equation}
    the consequence follows by Proposition \ref{prop: 2.7} with $f(x,\theta)=-2\partial_{\theta_i} b(x,\theta)$.

    ii) We have
    \begin{equation}
      \begin{aligned}
        \partial_{\theta_i} \partial_{\theta_j} \Phi_{n,\varepsilon,\ell}(\theta)
        =&\ \frac{2}{n} \sum_{k=\ell\vee1}^n A_\ell \partial_{\theta_i} b(\pmb{X}_{t_k:t_{k-\ell}}^{\varepsilon}, \theta)
        \cdot
        A_\ell \partial_{\theta_j} b(\pmb{X}_{t_k:t_{k-\ell}}^{\varepsilon}, \theta) \\
        &-2 \sum_{k=\ell\vee1}^n A_\ell \partial_{\theta_i} \partial_{\theta_j}b(\pmb{X}_{t_k:t_{k-\ell}}^{\varepsilon}, \theta)
        \cdot
        \left ( X_{t_k}^{\varepsilon} - X_{t_{k-1}}^{\varepsilon} - \frac{1}{n} A_\ell b(\pmb{X}_{t_k:t_{k-\ell}}^{\varepsilon}, \theta_0) \right ) .
      \end{aligned}
    \end{equation}
    By Proposition \ref{prop: 2.8}, the second term in the right-hand side converges in $P_{\theta_{\theta_0}}$ to zero uniformly in $\theta\in\Theta$.
    Also, by using Proposition \ref{prop: a.s. conv sum | Al f(X) |^q} with $f(x,\theta) = \partial_{\theta_i} b(x,\theta) \pm \partial_{\theta_j} b(x,\theta)$
    and $q=2$, the first term converges almost surely to $2I^{ij}(\theta)$ uniformly in $\theta\in\Theta$.
  \end{proof}

  \begin{proof}[Proof of Theorem \ref{thm: asymp dstr}]
    It follows from the mean value theorem that
    \begin{equation}\label{eq: first eq in thm 3.2}
      \varepsilon^{-1} \left ( \partial_{\theta_i} \Phi_{n,\varepsilon,\ell} (\hat{\theta}_{n,\varepsilon,\ell}) - \partial_{\theta_i} \Phi_{n,\varepsilon,\ell} (\theta_0) \right )
      = \varepsilon^{-1} ( \hat{\theta}_{n,\varepsilon,\ell} - \theta_0 )
      \cdot \int_0^1 \nabla_{\theta} \partial_{\theta_i} \Phi_{n,\varepsilon,\ell} \left ( \theta_0 + u ( \hat{\theta}_{n,\varepsilon,\ell} - \theta_0 ) \right ) du.
    \end{equation}
    By the consistency of $\hat{\theta}_{n,\varepsilon,\ell}$ and Proposition \ref{prop: 3.3} (i),
    the left-hand side converges to $2S_i(\theta_0)$ in $P_{\theta_0}$ as $n\to\infty$ and $\varepsilon\to0$
    if $\ell2^{4\ell}/n\to0$, $2^\ell\varepsilon\to0$ and $\ell2^{2\ell}/n\varepsilon\to0$.

    For an arbitrary convex neighborhood $U$ of $\theta_0\in \Theta_0$, we have
    \begin{multline}
      \left | \partial_{\theta_i} \partial_{\theta_j} \Phi_{n,\varepsilon,\ell}(\theta_0)
      - \int_0^1 \partial_{\theta_i} \partial_{\theta_j} \Phi_{n,\varepsilon,\ell} \left ( \theta_0
      + u ( \hat{\theta}_{n,\varepsilon,\ell} - \theta_0 ) \right ) du \right | \pmb{1}_{\{\hat{\theta}_{n,\varepsilon,\ell}\in U\}} \\
      \leq \sup_{\theta\in U} \left | \partial_{\theta_i} \partial_{\theta_j} \Phi_{n,\varepsilon,\ell}(\theta_0) -\partial_{\theta_i} \partial_{\theta_j} \Phi_{n,\varepsilon,\ell}(\theta) \right | \\
      \leq 2 \sup_{\theta\in U} \left | \partial_{\theta_i} \partial_{\theta_j} \Phi_{n,\varepsilon,\ell}(\theta) - 2 I^{ij}(\theta) \right | + 2 \sup_{\theta\in U} \left | I^{ij}(\theta) - I^{ij}(\theta_0) \right |.
    \end{multline}
    It follows from the consistency of $\hat{\theta}_{n,\varepsilon,\ell}$, Proposition \ref{prop: 3.3} (ii) and the continuity of $\theta\mapsto I(\theta)$ that
    if $2^\ell \varepsilon\to0$ and $\ell2^{2\ell}/n$ is bounded, then
    \begin{equation}
      2 I_{n,\varepsilon,\ell}^{ij}
      := \int_0^1 \partial_{\theta_i} \partial_{\theta_j} \Phi_{n,\varepsilon,\ell}\left(\theta_0 + u(\hat{\theta}_{n,\varepsilon,\ell}-\theta_0)\right) du
      \overset{P_{\theta_0}}{\longrightarrow} 2I^{ij}(\theta_0)
    \end{equation}
    as $n\to\infty$ and $\varepsilon\to0$.
    By Lemma \ref{lem: v_n=M_n w_n}, the proof is complete.
  \end{proof}

\section{Numerical experiment}

  In this section, we give a simulation by numerical computation to compare our estimators with well-known least squares estimators for an Ornstein-Uhlenbeck process given by
  \begin{equation}\label{eq: OU process}
    dX_t = - \theta_0 X_t dt + \varepsilon dB_t, \qquad X_0 = x_0,
  \end{equation}
  where $B$ is the standard Brownian motion.
  For simplicity, we set $\theta_0=1$ and $x_0=1$ with $\varepsilon=0.1, 0.5, 1.0$ and $n=50, 100, 1000$.
  We shall compare our Adams-Moulton type estimators
  \begin{equation}
    \hat{\theta}_{n,\varepsilon,\ell} := \argmin_{\theta\in\Theta}
    \sum_{k=\ell}^{n} \left | X_{t_k}^{\varepsilon} - X_{t_{k-1}}^{\varepsilon} + \frac{1}{n} \theta A_\ell b(\pmb{X}_{t_k:t_{k-\ell}}^\varepsilon)  \right |^2
    \quad (\ell=1,\dots,6)
  \end{equation}
  to the usual `Euler-type' estimator
  \begin{equation}
    \hat{\theta}_{n,\varepsilon} := \argmin_{\theta\in\Theta}
    \sum_{k=1}^{n} \left | X_{t_k}^{\varepsilon} - X_{t_{k-1}}^{\varepsilon} + \frac{1}{n} \theta X_{t_{k-1}}^{\varepsilon} \right |^2,
  \end{equation}
  where $A_\ell b(\pmb{X}_{t_k:t_{k-\ell}}^\varepsilon)$ with $b(x)=x$ are given by
  \begin{align}
    &A_\ell b(\pmb{X}_{t_k:t_{k-\ell}}^\varepsilon) \\
    &\quad=
    \begin{cases}
      \frac{1}{2}X_{t_{k}}^\varepsilon+\frac{1}{2}X_{t_{k-1}}^\varepsilon &\text{if}~\ell=1, \\
      \frac{5}{12}X_{t_{k}}^\varepsilon+\frac{2}{3}X_{t_{k-1}}^\varepsilon-\frac{1}{12}X_{t_{k-2}}^\varepsilon &\text{if}~\ell=2, \\
      \frac{3}{8}X_{t_{k}}^\varepsilon+\frac{19}{24}X_{t_{k-1}}^\varepsilon-\frac{5}{24}X_{t_{k-2}}^\varepsilon+\frac{1}{24}X_{t_{k-3}}^\varepsilon &\text{if}~\ell=3, \\
      \frac{251}{720}X_{t_{k}}^\varepsilon+\frac{323}{360}X_{t_{k-1}}^\varepsilon-\frac{11}{30}X_{t_{k-2}}^\varepsilon+\frac{53}{360}X_{t_{k-3}}^\varepsilon-\frac{19}{720}X_{t_{k-4}}^\varepsilon &\text{if}~\ell=4, \\
      \frac{95}{288}X_{t_{k}}^\varepsilon+\frac{1427}{1440}X_{t_{k-1}}^\varepsilon-\frac{133}{240}X_{t_{k-2}}^\varepsilon+\frac{241}{720}X_{t_{k-3}}^\varepsilon-\frac{173}{1440}X_{t_{k-4}}^\varepsilon+\frac{3}{160}X_{t_{k-5}}^\varepsilon &\text{if}~\ell=5, \\
      \frac{19087}{60480}X_{t_{k}}^\varepsilon+\frac{2713}{2520}X_{t_{k-1}}^\varepsilon-\frac{15487}{20160}X_{t_{k-2}}^\varepsilon+\frac{586}{945}X_{t_{k-3}}^\varepsilon\\
      \hspace{126pt}-\frac{6737}{20160}X_{t_{k-4}}^\varepsilon+\frac{263}{2520}X_{t_{k-5}}^\varepsilon-\frac{863}{60480}X_{t_{k-6}}^\varepsilon &\text{if}~\ell=6.
    \end{cases}
  \end{align}
  Such coefficients from the Adams-Moulton method can be seen, \textit{e.g.}, in Table 244 in Butcher \cite{butcher2016numerical}.
  Note that the Euler-type LSE $\hat{\theta}_{n,\varepsilon}$ is slightly different from the Adams-Moulton type LSE $\hat{\theta}_{n,\varepsilon,0}$, \textit{i.e.}, `backward Euler-type',
  but for the sake of both similarity, we omit to consider $\hat{\theta}_{n,\varepsilon,0}$.

  \begin{table}[t]
    \begin{center}
      \scriptsize
      \begin{minipage}{0.65\hsize}
        \caption{}\label{table: numerical experiment}
        Sample mean (with standard deviation in parentheses) of LSEs, based on 10,000 sample paths from the OU process \eqref{eq: OU process} with $(\theta_0, x_0)=(1.0, 1.0)$.
        We emphasize the best average of LSEs for each $(\varepsilon,n)$ using a bold font.
        \begin{center}
          \begin{tabular}{rccc}
            \hline
            $\varepsilon = 1.0$
                    & $n=10$ & $n=100$             & $n=1000$ \\
            \hline
            Euler & 1.663489 (1.471654) & 1.931070 (1.821037) & 1.966545 (1.873523) \\
            AM1   & 0.951550 (1.283734) & 0.802641 (1.038229) & 0.790167 (1.010930) \\
            AM2   & 1.028716 (1.564615) & \pmb{0.987162} (1.187243) & \pmb{0.986894} (1.152948) \\
            AM3   & 1.074215 (1.833185) & 1.078157 (1.261705) & 1.084002 (1.225335) \\
            AM4   & 1.087510 (2.127400) & 1.131508 (1.309929) & 1.146718 (1.273696) \\
            AM5   & \pmb{1.026782} (2.333814) & 1.167348 (1.354262) & 1.190314 (1.307555) \\
            AM6   & 0.884837 (2.779585) & 1.192481 (1.387997) & 1.224354 (1.336056) \\
            \hline
            \\
            \hline
            $\varepsilon = 0.1$
                    & $n=10$ & $n=100$             & $n=1000$ \\
            \hline
            Eular & 0.964199 (0.138219) & 1.010592 (0.151375) & 1.015443 (0.152811) \\
            AM1   & \pmb{1.004400} (0.154055) & \pmb{1.004306} (0.152030) & \pmb{1.004406} (0.151787) \\
            AM2   & 1.008660 (0.174779) & 1.006182 (0.153983) & 1.006442 (0.152153) \\
            AM3   & 1.012103 (0.198794) & 1.007291 (0.156226) & 1.007306 (0.152466) \\
            AM4   & 1.013956 (0.230109) & 1.007880 (0.157855) & 1.007970 (0.152656) \\
            AM5   & 1.015775 (0.269280) & 1.008047 (0.159778) & 1.008404 (0.152829) \\
            AM6   & 1.019984 (0.321273) & 1.008658 (0.162266) & 1.008764 (0.153157) \\
            \hline
            \\
            \hline
            $\varepsilon = 0.01$
                  & $n=10$ & $n=100$             & $n=1000$ \\
            \hline
            Eular & 0.951791 (0.013711) & 0.995199 (0.014975) & 0.999686 (0.015110) \\
            AM1   & 0.999246 (0.015337) & 1.000049 (0.015147) & \pmb{1.000074} (0.015126) \\
            AM2   & 1.000177 (0.017310) & 1.000061 (0.015320) & 1.000102 (0.015141) \\
            AM3   & 1.000232 (0.019645) & 1.000070 (0.015533) & 1.000100 (0.015160) \\
            AM4   & 1.000138 (0.022662) &  1.000062 (0.015683) & 1.000110 (0.015169) \\
            AM5   & \pmb{1.000017} (0.026460) & \pmb{1.000026} (0.015867) & 1.000113 (0.015182) \\
            AM6   & 1.000139 (0.031419) & 1.000041 (0.016101) & 1.000117 (0.015209) \\
            \hline
          \end{tabular}
        \end {center}
        AM$\ell$: LSE via the Adams-Moulton method with order $\ell$ ($\ell=1,\dots,6$).
      \end{minipage}
    \end{center}
  \end{table}

  In Table \ref{table: numerical experiment}, we compute
  \begin{equation}
    \hat{\theta}_{n,\varepsilon} = - \frac{\sum_{k=1}^n (X_{t_k}^\varepsilon-X_{t_{k-1}}^\varepsilon)X_{t_{k-1}}^\varepsilon}{\frac{1}{n} \sum_{k=1}^n |X_{t_{k-1}}^\varepsilon|^2},
  \end{equation}
  and
  \begin{equation}
    \hat{\theta}_{n,\varepsilon,\ell} = - \frac{\sum_{k=\ell}^n (X_{t_k}^\varepsilon-X_{t_{k-1}}^\varepsilon)A_\ell b(\pmb{X}_{t_k:t_k-\ell}^\varepsilon)}{\frac{1}{n} \sum_{k=\ell}^n |A_\ell b(\pmb{X}_{t_k:t_k-\ell}^\varepsilon)|^2} \quad (\ell=1,\dots,6),
  \end{equation}
  by using a sample path $\{X_{t_k}^\varepsilon\}_{k=0}^{n}$ made by
  \begin{equation}
    X_{t_0}^\varepsilon = x_0, \quad X_{t_k} = e^{-\theta_0 \Delta t} X_{t_{k-1}}
    + \varepsilon \sqrt{\frac{1-e^{-2\theta_0\Delta t}}{2\theta_0}} N(0,1), \quad
    \Delta t = t_k - t_{k-1} = \frac{1}{n}
  \end{equation}
  as a well-known way of constructing an exact numerical solution of \eqref{eq: OU process}, where $N(0,1)$ is the standard normal variable.
  We iterate this computation 10,000 times and show their sample means and standard deviations in Table \ref{table: numerical experiment}.
  We also plot the sample means and 95\% confidence intervals of $\hat{\theta}_{n,\varepsilon,\ell}$ through iterations in Figure \ref{fig: numerical experiment}.

  \begin{figure}[t]
    \begin{center}
      \includegraphics[width=0.49\hsize]{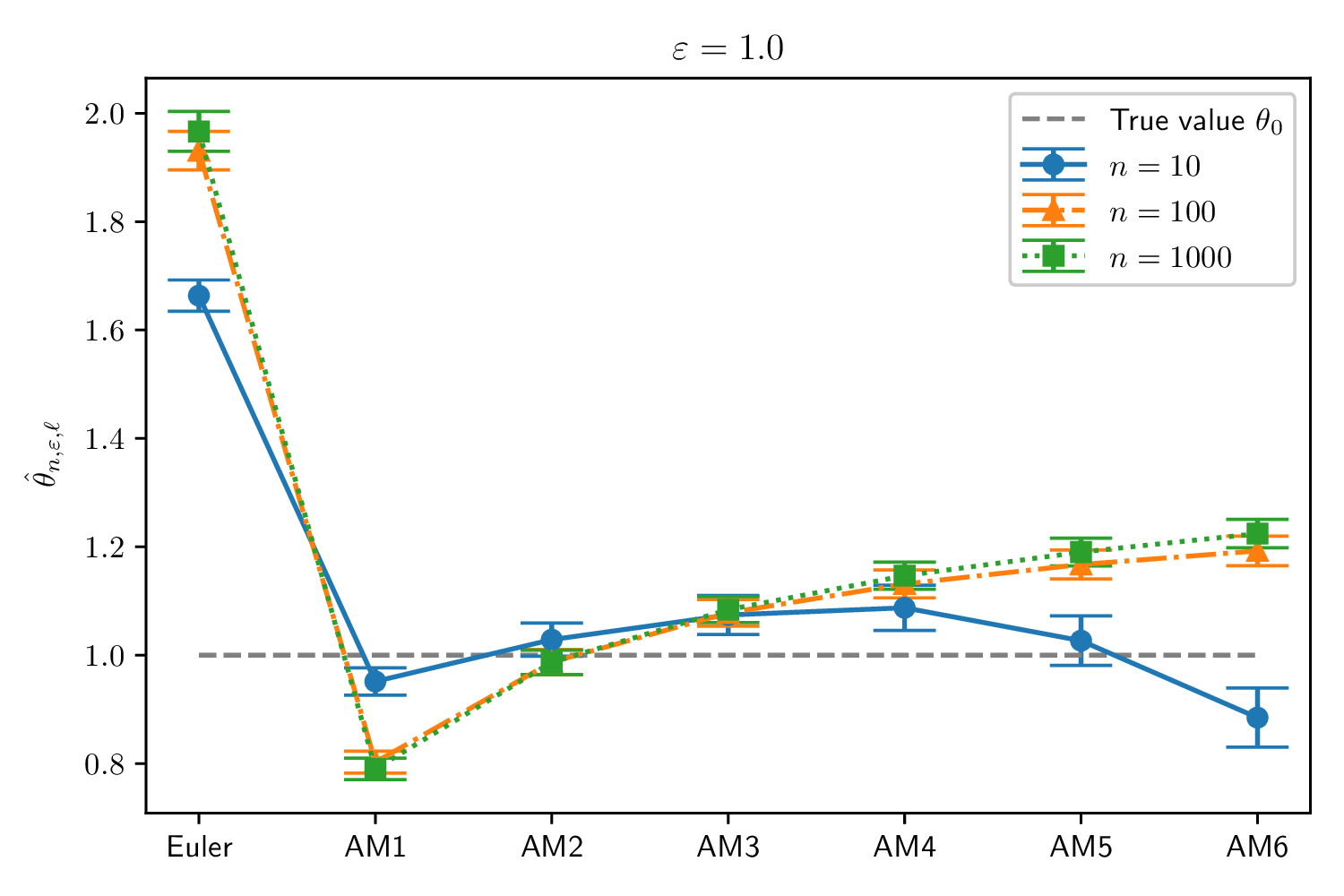}
      \includegraphics[width=0.49\hsize]{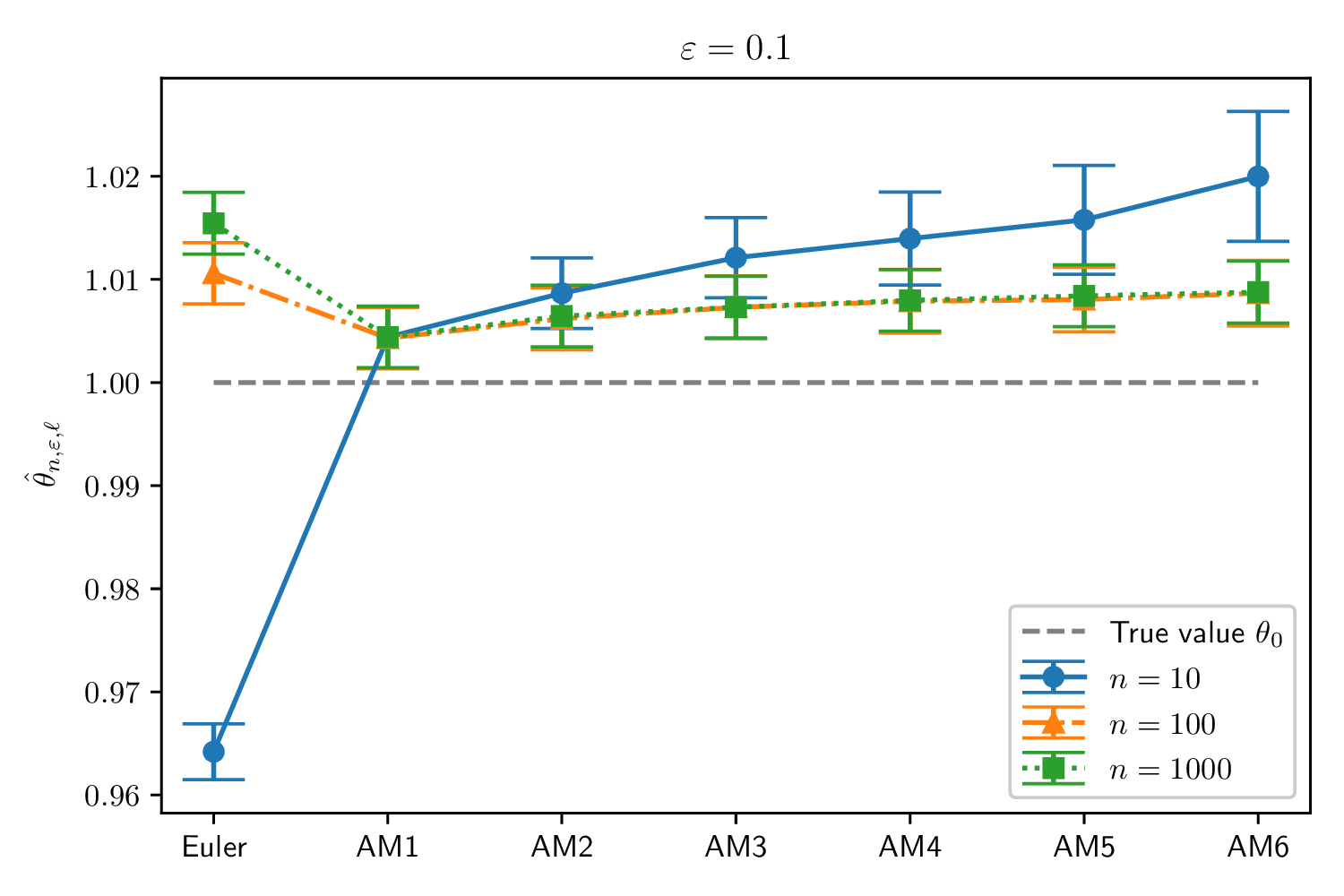}
      \includegraphics[width=0.49\hsize]{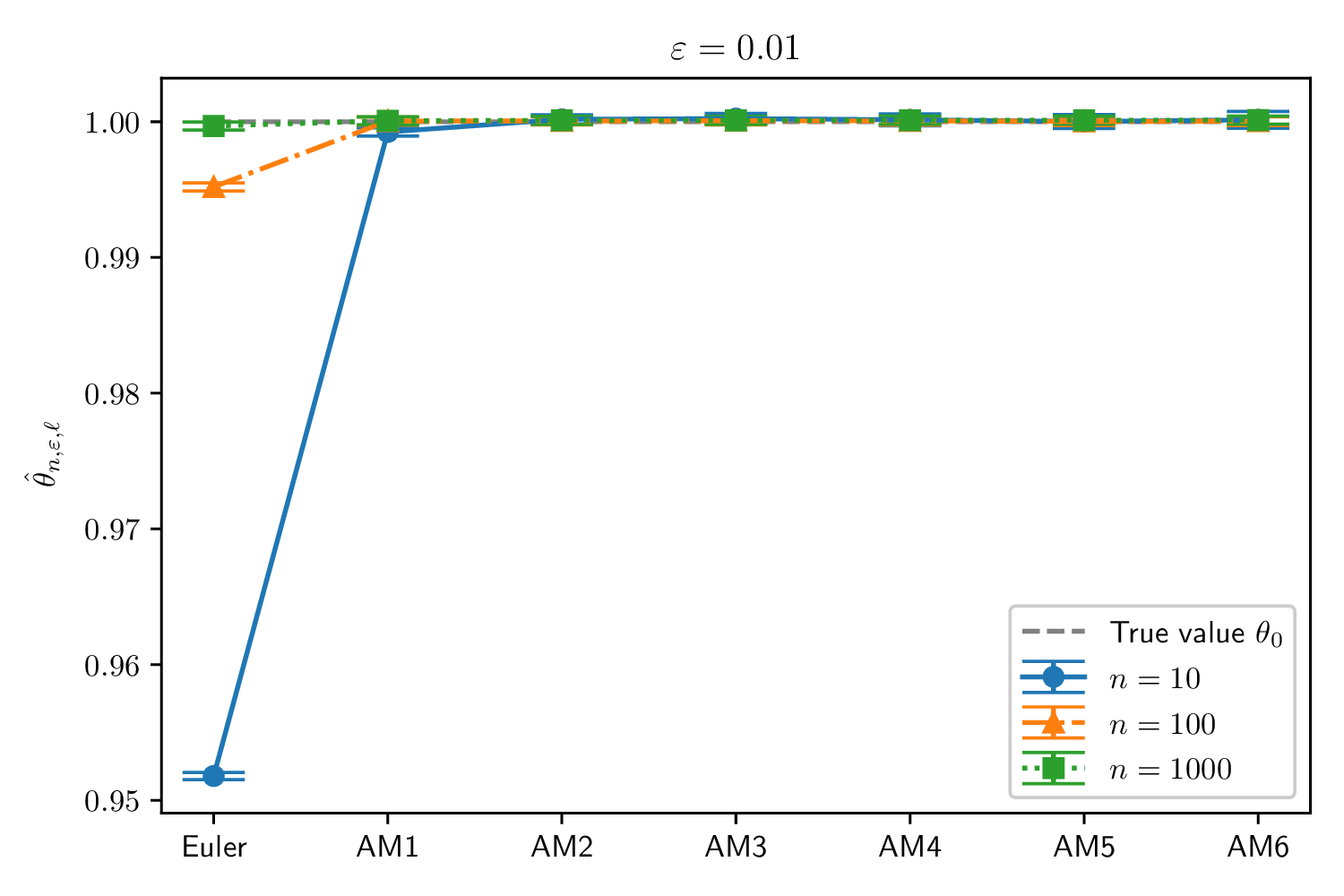}
      \caption{}\label{fig: numerical experiment}
      The means and 95\% confidence intervals through 10,000 \\
      iteration for $\hat{\theta}_{n,\varepsilon}$ (Euler)
      and $\hat{\theta}_{n,\varepsilon,\ell}$ (AM$\ell$, $\ell=1,\dots,6$).
    \end{center}
  \end{figure}

  \ \vspace{3mm}\\
  \begin{flushleft}
    {\bf\large Acknowledgements.}
    This research was partially supported by JSPS KAKENHI Grant-in-Aid for Scientific Research (A) \#17H01100 and JST CREST \#PMJCR14D7, Japan.
  \end{flushleft}

  \appendix

  \section{Appendix}

  \begin{lemma}\label{lem: L1-ineq for adams coef}
    Let $\gamma_{\ell\nu}$ and $\beta_{\ell\nu}$ be given by \eqref{eq: def of adams-bashforth} and \eqref{eq: def of adams-moulton}.
    Then,
    \begin{align}
      \sum_{\nu=1}^{\ell} \left | \gamma_{\ell\nu} \right | \leq \ell 2^{\ell-1} \quad (\ell=1,2,\dots), \qquad
      \sum_{\nu=0}^{\ell} \left | \beta_{\ell\nu} \right | \leq 2^{\ell} \quad (\ell=0,1,\dots).
    \end{align}
  \end{lemma}

  \begin{proof}
    The conclusion is obtained from
    \begin{equation}
      \sum_{\nu=1}^{\ell} \left | \gamma_{\ell\nu} \right |
      = \sum_{\nu=1}^{\ell} \frac{1}{(\nu-1)!(\ell-\nu)!} \int_0^1 \prod_{\substack{j=1 \\ j\neq\nu}}^\ell (u+j-1) \, du
      \leq \sum_{\nu=1}^{\ell} \frac{\ell!}{(\nu-1)!(\ell-\nu)!} = \ell 2^{\ell-1}
    \end{equation}
    for $\ell=1,2\dots$, and
    \begin{equation}
      \sum_{\nu=0}^{\ell} \left | \beta_{\ell\nu} \right |
      = \sum_{\nu=0}^{\ell} \frac{1}{\nu!(\ell-\nu)!} \int_0^1 \prod_{\substack{j=0 \\ j\neq\nu}}^\ell (u+j-1) \, du
      \leq \sum_{\nu=0}^{\ell} \frac{\ell!}{\nu!(\ell-\nu)!} = 2^{\ell}
    \end{equation}
    for $\ell=0,1,\dots$.
  \end{proof}

  \begin{lemma}\label{lem: ave int conv}
    Let $g$ be a continuous function on $\mathbb{R}^{d}$, let $t\mapsto y_t$ be an $\mathbb{R}^{d}$-valued continuous function on $[0,1]$, and let $\{f(\cdot,\theta)\}_{\theta\in\Theta}$ be a pointwise equicontinuous family of functions from $\mathbb{R}^{d}$ to $\mathbb{R}^d$.
    If $\ell/n\to0$ as $n\to\infty$,
    then
    \begin{equation}
      \frac{1}{n} \sum_{k=\ell\vee1}^{n} g \left ( \dashint_{t_{k-1}}^{t_k}f(y_t, \theta) \, dt \right )
      \to \int_0^1 g \circ f(y_t, \theta) \, dt
    \end{equation}
    as $n\to\infty$, uniformly in $\theta\in\Theta$.
  \end{lemma}

  \begin{proof}
    Since $\{f(y_\cdot,\theta)\}_{\theta\in\Theta}$ is uniformly equicontinuous on $[0,1]$,
    for any $\eta>0$ there exists $N\in\mathbb{N}$ such that $\theta\in\Theta$, $|s-t|\leq1/N$
    $\Rightarrow$ $\left | f(y_s,\theta) - f(y_t,\theta) \right | < \eta$.
    Then, for all $n\geq N$, $t\in[0,1)$ and $\theta\in\Theta$
    \begin{equation}
      \left | \sum_{k=1}^{n} \pmb{1}_{[t_{k-1},t_k)}(t) \, \dashint_{t_{k-1}}^{t_k}f(y_s, \theta) \, ds - f(y_t,\theta) \right |
      \leq \sum_{k=1}^{n} \pmb{1}_{[t_{k-1},t_k)}(t) \, \dashint_{t_{k-1}}^{t_k} \left | f(y_s, \theta) - f(y_t,\theta) \right | \, ds < \eta,
    \end{equation}
    and we have
    \begin{equation}
      \sum_{k=1}^{n} \pmb{1}_{[t_{k-1},t_k)}(t) \, \dashint_{t_{k-1}}^{t_k}f(y_s, \theta) \, ds
      \to f(y_t,\theta)
    \end{equation}
    uniformly in $(t,\theta)\in[0,1)\times\Theta$.
    By the continuity of $g$, we obtain
    \begin{equation}
      \begin{aligned}
        \frac{1}{n} \sum_{k=1}^n g \left ( \dashint_{t_{k-1}}^{t_k}f(y_t, \theta) \, dt \right )
        &= \int_0^1 \sum_{k=1}^n \pmb{1}_{[t_{k-1},t_k)}(t) \, g \left ( \dashint_{t_{k-1}}^{t_k}f(y_s, \theta) \, ds \right ) dt \\
        &= \int_0^1 g \left ( \sum_{k=1}^n \pmb{1}_{[t_{k-1},t_k)}(t) \, \dashint_{t_{k-1}}^{t_k}f(y_s, \theta) \, ds \right ) dt
        \to \int_0^1 g \circ f(y_t, \theta) \, dt
      \end{aligned}
    \end{equation}
    as $n\to\infty$, uniformly in $\theta\in\Theta$.
    Since $\{g\circ f(y_{\cdot},\theta)\}_{\theta\in \Theta}$ is equicontinuous at $t=0$,
    for $\ell\geq2$,
    \begin{equation}
      \frac{1}{n} \sum_{k=1}^{\ell-1} g \left ( \dashint_{t_{k-1}}^{t_k}f(y_t, \theta) \, dt \right )\to0
    \end{equation}
    as $\ell/n\to0$, uniformly in $\theta\in\Theta$.
  \end{proof}

  Let $(\Omega,P,\mathscr{F})$ be a probability space, and let $\Sym_{p}(\mathbb{R})$ denote the set of all $p\times p$ symmetric matrix with real entries and with the Frobenius norm $\|\cdot\|_{F}$.

  \begin{lemma}\label{lem: v_n=M_n w_n}
    Suppose that $v_n\overset{p}\to v$ in $\mathbb{R}^p$ and $M_n\overset{p}\to M$ in $\Sym_p(\mathbb{R})$ as $n\to\infty$, $w_n$ satisfies $v_n=M_n w_n$.
    If $M$ is positive definite, $w_n\overset{p}\to M^{-1}v$.
  \end{lemma}

  \begin{proof}
    Let $\eta$ be an arbitrary positive number less than the smallest eigenvalue of $M$.
    If $\|M_n-M\|_{F}<\eta$, then $0\prec M-\eta\mathbb{I}_{p\times p}\prec M_n\prec M+\eta\mathbb{I}_{p\times p}$,
    where $\mathbb{I}_{p\times p}$ is the identity matrix of size $p$ and $\prec$ is the Loewner order.
    This implies that $M_n$ is invetible and
    \begin{equation}
      \left ( M + \eta \mathbb{I}_{p\times p} \right )^{-1} \prec M_n^{-1} \prec \left ( M - \eta \mathbb{I}_{p\times p} \right )^{-1}.
    \end{equation}
    Since $(M\pm \eta \mathbb{I}_{p\times p})^{-1}\to M^{-1}$ in $\Sym_{p}(\mathbb{R})$ as $\eta\to0$, there exists a positive number $\tilde{\eta}$ depending only on $M, p$ and $\eta$ such that $\|M_n^{-1}-M^{-1}\|_F<\tilde{\eta}$ and $\tilde{\eta} \to 0$ as $\eta\to0$.

    Set $\mathscr{D}_n:=\{\omega\in\Omega \, | \, M_n(\omega)~\text{is invertible}\}$.
    Then, if an arbitrary positive number $\tilde{\eta}$ is sufficiently small, for some $\eta>0$ we have
    \begin{equation}
      P(\mathscr{D}_n^\mathrm{C}) + P(\pmb{1}_{\mathscr{D}_n}\|M_n^{-1}-M^{-1}\|_F>\tilde{\eta})
      \leq 2 P (\|M_n^{-1}-M^{-1}\|_F>\eta)
      \to 0,
    \end{equation}
    where $\pmb{1}_A$ is the indicator function on a set $A\subset\Omega$.
    Hence, we obtain
    \begin{equation}
      w_n = M_n^{-1}v_n \pmb{1}_{\mathscr{D}_n} + w_n \pmb{1}_{\mathscr{D}_n^\mathrm{C}}
      \overset{p}\to M^{-1}v
    \end{equation}
    as $n\to\infty$.
  \end{proof}

\end{document}